\documentclass[12pt, reqno]{amsart}
\usepackage{amsmath, amsthm, amscd, amsfonts, amssymb, graphicx, color}
\usepackage[bookmarksnumbered, colorlinks, plainpages]{hyperref}
\hypersetup{colorlinks=true,linkcolor=red, anchorcolor=green, citecolor=cyan, urlcolor=red, filecolor=magenta, pdftoolbar=true}

\newtheorem{theorem}{Theorem}[section]
\newtheorem{lemma}[theorem]{Lemma}
\newtheorem{proposition}[theorem]{Proposition}
\newtheorem{corollary}[theorem]{Corollary}
\theoremstyle{definition}
\newtheorem{definition}[theorem]{Definition}
\newtheorem{example}[theorem]{Example}

\theoremstyle{remark}
\newtheorem{remark}[theorem]{Remark}
\numberwithin{equation}{section}

\begin{document}

\setcounter{page}{1}

\title[Calder\'on-Lozanovskii interpolation]{Calder\'on-Lozanovskii interpolation on quasi-Banach lattices}

\author[Y. Raynaud, \MakeLowercase{and} P. Tradacete]{Yves Raynaud,$^1$ \MakeLowercase{and} Pedro Tradacete$^2$}

\address{$^{1}$Institut de Math\'ematiques de Jussieu-Paris Rive Gauche, CNRS and UPMC-Univ.~Paris-06, case
186,  75005 Paris, France}
\email{\textcolor[rgb]{0.00,0.00,0.84}{yves.raynaud@upmc.fr}}

\address{$^{2}$Department of Mathematics\\ Universidad Carlos III de Madrid\\ 28911, Legan\'es, Madrid, Spain.}
\email{\textcolor[rgb]{0.00,0.00,0.84}{ptradace@math.uc3m.es}}

\let\thefootnote\relax


\begin{abstract}
We consider the Calder\'on-Lozanovskii construction $\varphi(X_0,X_1)$ in the context of quasi-Banach lattices and provide an extension of a result by V. I. Ovchinnikov concerning the associated interpolation methods $\varphi^c$ and $\varphi^0$. Our approach is based on the interpolation properties of $(\infty,1)$-regular operators between quasi-Banach lattices.
\end{abstract} \maketitle

\section{Introduction}
The aim of this note is to study the interpolation properties of the Calder\'on-Lozanovskii construction in the quasi-Banach lattice setting. Let us start by recalling this construction: Given $(X_0,X_1)$ a compatible pair of quasi-Banach lattices and a function $\varphi:\mathbb{R}_+^2\rightarrow \mathbb{R}_+$ which is homogeneous and non-decreasing in each argument, we consider the space $\varphi(X_0,X_1)$ of those $x\in X_0+X_1$ such that $|x|\leq \varphi(x_0,x_1)$  for some $x_0\in X_0$ and $x_1\in X_1$. This space becomes a quasi-Banach lattice when endowed with the quasi-norm
$$
\|x\|_{\varphi(X_0,X_1)}=\inf\{\lambda>0:|x|\leq \lambda\varphi(x_0,x_1),\,\|x_0\|_{X_0}\leq1,\|x_1\|_{X_1}\leq1\}.
$$

This space was introduced by G. Ya. Lozanovskii and studied in \cite{Lozanovskii4} (see also the references therein). In particular, a lot of work has been done for the case of $\varphi(s,t)=s^{1-\theta} t^{\theta}$ for some $\theta\in(0,1)$, which yields the Calder\'on product $X_0^{1-\theta}X_1^\theta$ (see \cite{Calderon}). The relation between this and the complex interpolation methods has been carefully investigated in the literature (see \cite{Calderon,KM,RT,Sestakov}).

There is an obvious interest in extending interpolation results which are valid in the Banach space, or Banach lattice, setting to the more general context of quasi-Banach spaces (see for instance \cite{CMS, GM, Kalton2, M}).

Our interest in this note is to relate the construction $\varphi(X_0,X_1)$ with two well-known interpolation functors. In this respect, recall that given quasi-normed spaces $X$ and $Y$, such that there is a continuous inclusion $i:X\hookrightarrow Y$, the Gagliardo completion of $X$ in $Y$ is the quasi-normed space whose unit ball is the closure of $i(B_X)$ in $Y$, where as usual $B_X$ denotes the unit ball of $X$; note that when $Y$ is complete, this clearly defines a quasi-Banach space. Let us denote $\varphi^c(X_0,X_1)$ the Gagliardo completion of the space $\varphi(X_0,X_1)$ in $X_0+X_1$. Also, let $\varphi^0(X_0,X_1)$ denote the closure of the intersection $X_0\cap X_1$ in $\varphi(X_0,X_1)$. We obviously have the following bounded inclusions:
 $$
 \varphi^0(X_0,X_1)\subset \varphi(X_0,X_1)\subset \varphi^c(X_0,X_1).
 $$

It was proved by V. I. Ovchinnikov that $\varphi^0$ and $\varphi^c$ are interpolation functors in the category of Banach lattices of measurable functions (see \cite{Ovchi} and \cite[Theorem 4.3.11]{BK}). Earlier attempts to extend these interpolation functors to the category of quasi-Banach lattices have been made by P. Nilsson \cite{Nilsson} and V. I. Ovchinnikov \cite{Ovchi2}.

Our main result in this paper is the extension of this fact to the category of quasi-Banach lattices with the $K_{\infty,1}$ property: i.e. those spaces $X$ for which the following inequality holds
$$
\big\|\max_{1\leq i\leq n}|x_i|\big\|\leq C\max_{|a_i|\leq1}\Big\|\sum_{i=1}^n a_ix_i\Big\|,
$$
for some constant $C>0$ independent of $(x_i)_{i=1}^n\subset X$ (see Section \ref{Krivine property} below). It should be noted that a large class of quasi-Banach lattices, namely that of $L$-convex quasi-Banach lattices, introduced by N. Kalton in \cite{kalton}, have the $K_{\infty,1}$ property (see also \cite{Nilsson}, in connection with the interpolation of $L$-convex lattices).

An important ingredient in our proof will be the class of $(p,q)$-regular operators, i.e. those satisfying estimates of the form
$$
\bigg\|\bigg(\sum_{i=1}^n|Tx_i|^p\bigg)^{\frac1p}\bigg\|\leq K \bigg\|\bigg(\sum_{i=1}^n|x_i|^q\bigg)^{\frac1q}\bigg\|.
$$
This class of operators was introduced by A. V. Bukhvalov in \cite{B1}, where some interpolation results between Banach lattices were obtained. It will be shown in Theorem \ref{gagliardo} that $(\infty,1)$-regular operators have good interpolation properties with respect to the Calder\'on-Lozanovskii construction. This fact will allow us to extend further the interpolation functors $\varphi^c$ and $\varphi^0$.

\section{Definitions and preliminaries}

Let $\mathbb R_+=\{x\in\mathbb R:x\geq 0\}$. Recall that a quasi-Banach space $(X,\|\cdot\|)$ is a vector space which is complete for the metric induced by the quasi-norm $\|\cdot\|:X\rightarrow \mathbb R_+$, that satisfies
\begin{align*}
&\|x\|=0\Leftrightarrow x=0\\
&\|\lambda x\|=|\lambda|\|x\|\\
&\|x+y\|\leq C(\|x\|+\|y\|)
\end{align*}
where $C\geq1$ is independent of $x,y\in X$. If moreover, $X$ is a vector lattice with $\|x\|\leq\|y\|$ whenever $|x|\leq |y|$, then we say that $X$ is a quasi-Banach lattice.

We will denote by $\mathcal P$ the set of all functions $\varphi:(0,\infty)\times (0,\infty)\rightarrow \mathbb{R}_+$ satisfying \begin{align*}
&\varphi(\lambda s,\lambda t)=\lambda\varphi(s,t) \textrm{ for every }s,t,\lambda>0\\
&\varphi(\cdot,t)\textrm{ is non-decreasing for every } t>0,\\
&\varphi(s,\cdot)\textrm{ is non-decreasing for every } s>0.
\end{align*}
We will usually make the normalization $\varphi(1,1)=1$. Given $\varphi\in\mathcal P$, let us denote $\varphi_0(t)=\varphi(t,1)$ and $\varphi_1(t)=\varphi(1,t)$. Note that
\[\varphi_1(t)=t\varphi_0(1/t)\]
It follows that both $\varphi_0$ and $\varphi_1$ are quasi-concave functions (i.e., $\varphi_i(t)$ is non-decreasing and $\varphi_i(t)/t$ is non-increasing, for $i=0,1$). We will make repeated use of the fact that every quasi-concave function is equivalent, up to a universal constant, to a concave function (cf. \cite[Corollary 3.1.4]{BK}). For $0<s<t$ we have
\[ \varphi_i(s)\le\varphi_i(t)\le \frac ts\varphi_i(s)\]
thus $\varphi_i$ is continuous on $(0,\infty)$. It follows from the equations
\[\varphi(s,t)=t\varphi_0(s/t)=s\varphi_1(t/s) \]
that $\varphi$ is continuous on $(0,\infty)\times (0,\infty)$.
 Since $\varphi_i$ is increasing, it has a right limit $\varphi_i(0^+)$ at $0$ and thus has a continuous extension $\bar\varphi_i$ to $\mathbb R_+$. Let us extend $\varphi$ to a  function $\bar\varphi$ on $\mathbb R_+^2$ by setting
\[\bar\varphi(s,0)=s\varphi_1(0^+)\hbox{ and } \bar\varphi(0,t)= t\varphi_0(0^+)\]
This extension is continuous. Indeed, since $\bar\varphi(s,t)=s\bar \varphi_1(t/s)$ for $s>0, t\ge 0$ (resp. $\bar\varphi(s,t)=t\bar \varphi_0(s/t)$ for $s\ge 0, t> 0$) $\bar\varphi$ is continuous on $\mathbb R_+^2\setminus\{(0,0)\}$; moreover from $\bar\varphi(s,t)\le (s\vee t) \varphi(1,1)$ it follows that $\bar\varphi$ is also continuous at $(0,0)$. We shall from now on denote simply by $\varphi$ the unique continuous extension of $\varphi$ to $\mathbb R_+^2$.

Given quasi-Banach lattices $X_0,X_1$, we say that $(X_0,X_1)$ is a compatible pair of quasi-Banach lattices when there exist a (Hausdorff, locally solid) topological vector lattice $X$, and inclusions $j_i:X_i\hookrightarrow X$ which are continuous, interval preserving, lattice homomorphisms, for $i=0,1$. In this way, the space
$$
X_0+X_1=\{x\in X: x=x_0+x_1,\,\textrm{with}\,x_0\in X_0,\, x_1\in X_1\}
$$
becomes a quasi-Banach lattice, endowed with the quasi-norm
$$
\|x\|=\inf\{\|x_0\|_{X_0}+\|x_1\|_{X_1}:x=x_0+x_1\},
$$
which contains $X_0$ and $X_1$ as (non-closed) ideals. 

Note that this setting is more general than the one considered in \cite{BK} (where $X$ is the space of measurable functions over some measure space) or in \cite{Lozanovskii4} (where $X$ is a $C_\infty(Q)$-space, i.e. the space of extended continuous scalar functions with dense domain over a Stonean compact space $Q$). In particular, $X_0$ and $X_1$ need not to be order complete.

Now, given a compatible pair of quasi-Banach lattices $(X_0,X_1)$ and a function $\varphi\in\mathcal P$, let us consider the Calder\'on-Lozanovskii space \cite{Lozanovskii1,Lozanovskii4}:
$$
\varphi(X_0,X_1)=\{x\in X_0+X_1:|x|\leq\varphi(x_0,x_1)\,\textrm{for some }x_0\in X_0^+,\,x_1\in X_1^+\}.
$$
Here, for any pair of positive elements $x_0,x_1$ in a quasi-Banach lattice, $\varphi(x_0,x_1)$ is defined in an unambiguous way by means of Krivine's functional calculus for continuous positively 1-homogeneous functions on $\mathbb R^2$ (see \cite[pp. 40--42]{LT2}, \cite{popa}). Indeed, $\varphi$ may be extended to such a function (e.g., $\hat\varphi(s,t)=\varphi(s\vee 0,t\vee 0)$).

The space $\varphi(X_0,X_1)$ is a quasi-Banach lattice equipped with the quasi-norm
$$
\|x\|_{\varphi(X_0,X_1)}=\inf\{\lambda>0:|x|\leq \lambda\varphi(x_0,x_1),\,\|x_0\|_{X_0}\leq1,\|x_1\|_{X_1}\leq1\}.
$$
Actually, we have
$$
\|x+y\|_{\varphi(X_0,X_1)}\leq\max\{C_0,C_1\}(\|x\|_{\varphi(X_0,X_1)}+\|y\|_{\varphi(X_0,X_1)}),
$$
where $C_i$ is the constant appearing in the triangle inequality corresponding to $X_i$ ($i=0,1$).

Given a function $\varphi$ as above, there is a natural decomposition into piecewise linear functions due to Y. A. Brudnyi and N. Y. Kruglyak (see \cite[Proposition 3.2.5]{BK}, or \cite{KMP}). We present next a small modification of this construction which is more suitable to our purposes.

\begin{lemma}\label{l:quasi-concave}
Let $\varphi\in \mathcal P$. Given $q>1$, there exist $M,N\in\mathbb N\cup\{\infty\}$, extended sequences $(t_k)_{k=-2M}^{2N}\subset[0,+\infty]$, and $(\varepsilon_k)_{k=-M}^N\subset[0,1]$ satisfying the following properties:
\begin{enumerate}
\item $(t_k)_{k=-2M}^{2N}$ is increasing, $0<\varepsilon_k<\min\{t_{2k}-t_{2k-1},t_{2k+3}-t_{2k+2}\}$.
\item\label{sum phi min} For every $s,t\in(0,+\infty)$ it holds that
$$
\sum_{k=-M}^N\varphi(1,t_{2k+1})\min\Big(s,\frac{t}{t_{2k+1}}\Big)\leq\frac{q+1}{q-1}\varphi(s,t).
$$
\item\label{star} for all $t\in[t_{2k}-\varepsilon_k,t_{2k+2}+\varepsilon_k]$
$$
\varphi(1,t)\leq q\varphi(1,t_{2k+1})\min\Big(1,\frac{t}{t_{2k+1}}\Big).
$$
\end{enumerate}
The notation here is consistent in the following sense:
\begin{itemize}
\item If $M=\infty$ then $\lim_{k\rightarrow-\infty}t_k=0=\lim_{k\rightarrow-\infty}\varepsilon_k$.
\item If $N=\infty,$ then $\lim_{k\rightarrow+\infty}t_k=+\infty,$ $\lim_{k\rightarrow+\infty}\varepsilon_k=0.$
\item If both $M,N$ are finite, then $t_{-2M}=0$, $t_{2N}=+\infty$, $\varepsilon_{-M}=\varepsilon_N=0$.
\end{itemize}
\end{lemma}

\begin{proof}
We work with the function $\varphi_1(t)=\varphi(1,t)$. Since $\varphi_1$ is quasi-concave, for every $s,t\in\mathbb{R}_+$, we have
$$
\varphi_1(t)\leq\max\big(1,\frac{t}{s}\big)\varphi_1(s).
$$
Thus, we can assume without loss of generality that $\varphi_1$ is a continuous concave function on $\mathbb R_+$ (cf. \cite[Corollary 3.1.4]{BK}).

According to \cite[Proposition 3.2.5]{BK}, for any $q'\in(1,q)$ there exist $M,N\in\mathbb N\cup\{\infty\}$ and an increasing sequence $(t_k)_{k=-2M}^{2N}\subset[0,+\infty]$, satisfying the following properties:
\begin{enumerate}
\item[(a)] If $M,N<\infty$, then $t_{-2M}=0$ and $t_{2N}=+\infty$. Otherwise, if $M=\infty$, then $\lim_{k\rightarrow-\infty}t_k=0,$ while if $N=\infty,$ then $\lim_{k\rightarrow+\infty}t_k=+\infty.$
\item[(b)] For $-M\leq k\leq N$ we have
$$
\frac{\varphi_1(t_{2k})}{t_{2k}}=q'\frac{\varphi_1(t_{2k+1})}{t_{2k+1}}\hspace{1cm}\textrm{ and }\hspace{1cm}\varphi_1(t_{2k+2})=q'\varphi_1(t_{2k+1}).
$$
\item[(c)] For every $s,t\in(0,+\infty)$ it holds that
$$
\sum_{k=-M}^N\varphi_1(t_{2k+1})\min\Big(s,\frac{t}{t_{2k+1}}\Big)\leq\frac{q'+1}{q'-1}\varphi(s,t).
$$
\end{enumerate}

Note that (b) yields that for $t\in[t_{2k},t_{2k}+2]$ one has
$$
\varphi_1(t)\leq q'\varphi_1(t_{2k+1})\min\Big(1,\frac{t}{t_{2k+1}}\Big).
$$
Now, for any $\varepsilon\in(0,\frac{q}{q'}-1)$, using the continuity of $\varphi_1$ we can find a sequence $(\varepsilon_k)$ with $\lim_{|k|\rightarrow+\infty}\varepsilon_k=0$,
$$
0<\varepsilon_k<\min\{t_{2k}-t_{2k-1},t_{2k+3}-t_{2k+2}\},
$$
and such that
$$
\varphi_1(t)\leq (1+\varepsilon)q'\varphi_1(t_{2k+1})\min\Big(1,\frac{t}{t_{2k+1}}\Big),
$$
for all $t\in[t_{2k}-\varepsilon_k,t_{2k+2}+\varepsilon_k]$. These sequences satisfy the required properties.
\end{proof}

Throughout, we will be using the usual local representation of a quasi-Banach lattice via $C(\Omega)$ spaces (see \cite{popa}): that is, given a positive element in a quasi-Banach lattice $e\in X$, the (non-closed) ideal generated by $e$ is isomorphic to a space $C(\Omega)$, for a certain compact Hausdorff space $\Omega$, and we can consider an injective lattice homomorphism $J:C(\Omega)\rightarrow X$ such that $J(1_\Omega)=e$ and $J(B_{C(\Omega)})=[-e,e]$.

Let us briefly recall the formal meaning of an interpolation functor between quasi-Banach lattices. We use the terminology of category theory as in \cite[2.3]{BK}. Let $\mathcal{QBL}$ denote the category of quasi-Banach lattices and bounded linear operators between them, and $\overset{\longrightarrow}{\mathcal{QBL}}$ the category of compatible pairs $\overset{\rightarrow}X=(X_0,X_1)$ of quasi-Banach lattices and linear operators between them, where a linear operator
$$
T:\overset{\rightarrow}X\rightarrow \overset{\rightarrow}Y
$$
is a bounded linear mapping $T:X_0+X_1\rightarrow Y_0+Y_1$ satisfying $T|_{X_0}:X_0\rightarrow Y_0$ and $T|_{X_1}:X_1\rightarrow Y_1$ (both being bounded too).

A functor $F:\overset{\longrightarrow}{\mathcal{QBL}}\rightarrow \mathcal{QBL}$ is called an interpolation functor if:
\begin{enumerate}
\item[(i)] For every $\overset{\rightarrow}X=(X_0,X_1)$, we have bounded inclusions $X_0\cap X_1\hookrightarrow F(\overset{\rightarrow}X)\hookrightarrow X_0+X_1$.
\item[(ii)] For every $T:\overset{\rightarrow}X\rightarrow \overset{\rightarrow}Y$, the operator $F(T)=T|_{F(\overset{\rightarrow}X)}:F(\overset{\rightarrow}X)\rightarrow F(\overset{\rightarrow}Y)$ is bounded.
\end{enumerate}

In particular, this implies that $F(\overset{\rightarrow}X)$ is an interpolation space for every $\overset{\rightarrow}X$.

\section{Interpolation of $(\infty,1)$-regular operators}

Given quasi-Banach lattices $E,F$, and $1\leq p,q< \infty$ a linear operator $T:E\rightarrow F$ is called \emph{$(p,q)$-regular} if there is a constant $K>0$ such that for every $\{x_i\}_{i=1}^n\subset E$
$$
\bigg\|\bigg(\sum_{i=1}^n|Tx_i|^p\bigg)^{\frac1p}\bigg\|\leq K \bigg\|\bigg(\sum_{i=1}^n|x_i|^q\bigg)^{\frac1q}\bigg\|.
$$
Similarly, $T$ will be called $(p,\infty)$-regular (respectively, $(\infty,q)$ regular) when
$$
\bigg\|\bigg(\sum_{i=1}^n|Tx_i|^p\bigg)^{\frac1p}\bigg\|\leq K \bigg\|\bigvee_{i=1}^n|x_i|\bigg\|, \hspace{1cm}\Bigg( \textrm{resp. } \bigg\|\bigvee_{i=1}^n|Tx_i|\bigg\|\leq K  \bigg\|\bigg(\sum_{i=1}^n|x_i|^q\bigg)^{\frac1q}\bigg\|.\Bigg)
$$
We will denote by $\rho_{p,q}(T)$ the smallest $K>0$ for which the above inequalities hold for arbitrary elements in $E$.

The class of $(p,q)$-regular operators was introduced in \cite{B1} (see also \cite{B2,Kusraev}), and has obvious connections with convexity and concavity (cf. \cite[1.d]{LT2}). It is clear that a $(p,q)$-regular operator $T$ is always bounded and $\|T\|\leq\rho_{p,q}(T)$. Also, if $T$ is $(p,q)$-regular, then it is $(p',q')$-regular for every $p'\geq p$ and $q'\leq q$, and moreover $\rho_{p',q'}(T)\leq\rho_{p,q}(T)$. In particular, among these, the largest class is that of $(\infty,1)$-regular operators, which satisfy
$$
\bigg\|\bigvee_{i=1}^n|Tx_i|\bigg\|\leq K  \bigg\|\sum_{i=1}^n|x_i|\bigg\|.
$$

If $F$ is Dedekind complete and $T:E\rightarrow F$ is a regular operator (i.e., $T$ can be written as a difference of two positive operators), then it is $(p,p)$-regular for every $1\leq p\leq \infty$, and $\rho_{p,p}(T)\leq\||T|\|$. In the converse direction, if $F$ is complemented by a positive projection in its bidual, then every $(1,1)$-regular operator $T:E\rightarrow F$ is regular \cite[p. 307]{Kusraev}.

In Section \ref{Krivine property}, we will consider spaces in which every linear operator is $(p,q)$-regular. In particular, an application of Grothendieck's inequality yields that every bounded linear operator between Banach lattices, or even L-convex quasi-Banach lattices, is $(2,2)$-regular.

We state now our main result concerning the interpolation of $(\infty,1)$-regular operators with respect to the functor $\varphi^c$.

\begin{theorem}\label{gagliardo}
Let $(X_0,X_1)$ and $(Y_0,Y_1)$ be compatible pairs of quasi-Banach lattices and $T:X_0+X_1\rightarrow Y_0+Y_1$ be a bounded operator such that $T|_{X_i}:X_i\rightarrow Y_i$ is $(\infty,1)$-regular for $i=0,1$. Then, for $\varphi\in\mathcal P$ we have that $T:\varphi^c(X_0,X_1)\rightarrow\varphi^c(Y_0,Y_1)$ is $(\infty,1)$-regular with
$$
\rho_{\infty,1}(T|_{\varphi^c(X_0,X_1)})\leq C\max\{\rho_{\infty,1}(T|_{X_0}),\rho_{\infty,1}(T|_{X_1})\},
$$
for some $C>0$ which only depends on $X_0,\,X_1,\,Y_0,\,Y_1$ and $\varphi$.
\end{theorem}

Before giving our proof, we need some preliminaries:

\begin{lemma}\label{l:sumregular}
Let $(X_0,X_1)$ and $(Y_0,Y_1)$ be interpolation couples of quasi-Banach lattices and $T:X_0+X_1\rightarrow Y_0+Y_1$ be a bounded operator such that $T|_{X_i}:X_i\rightarrow Y_i$ is $(\infty,1)$-regular for $i=0,1$. Then $T:X_0+X_1\rightarrow Y_0+Y_1$ is $(\infty,1)$-regular with
$$
\rho_{\infty,1}(T)\leq 2\max\{\rho_{\infty,1}(T|_{X_0}),\rho_{\infty,1}(T|_{X_1})\}.
$$
\end{lemma}

\begin{proof}
Let us consider $(z_i)_{i=1}^n\subset X_0+X_1$ such that $\|\sum_{i=1}^n|z_i|\|_{X_0+X_1}<1$. Hence, there exist positive $u\in X_0$, $v\in X_1$ with $\|u\|_{X_0}+\|v\|_{X_1}<1$ and
$$
\sum_{i=1}^n|z_i|\leq u+v.
$$
Using the Riesz decomposition property (cf. \cite[Theorem 1.1.1.viii]{M-N}), we can write $z_i=u_i+v_i$ for $i=1,\ldots,n$, with $\sum_{i=1}^n|u_i|\leq 2u$, $\sum_{i=1}^n|v_i|\leq 2v$. Now, since $T|_{X_j}$ is $(\infty,1)$-regular for $j=0,1$, we have that
\begin{eqnarray*}
\Big\|\bigvee_{i=1}^n|Tu_i|\Big\|_{Y_0}\leq \rho_{\infty,1}(T|_{X_0})\Big\|\sum_{i=1}^n|u_i|\Big\|_{X_0}\leq 2\rho_{\infty,1}(T|_{X_0})\|u\|_{X_0},\\
\Big\|\bigvee_{i=1}^n|Tv_i|\Big\|_{Y_1}\leq \rho_{\infty,1}(T|_{X_1})\Big\|\sum_{i=1}^n|v_i|\Big\|_{X_1}\leq 2\rho_{\infty,1}(T|_{X_1})\|v\|_{X_1}.
\end{eqnarray*}
These, together with
$$
\bigvee_{i=1}^n|Tz_i|\leq \bigvee_{i=1}^n|Tu_i|+\bigvee_{i=1}^n|Tv_i|
$$
yield that
$$
\Big\|\bigvee_{i=1}^n|Tz_i|\Big\|_{Y_0+Y_1}\leq 2\max\{\rho_{\infty,1}(T|_{X_0}),\rho_{\infty,1}(T|_{X_1})\}.
$$
This finishes the proof.
\end{proof}

\begin{lemma}\label{l:approximation}
There is a constant $\gamma>0$ such that given $(X_0,X_1)$, $(Y_0,Y_1)$, $T:X_0+X_1\rightarrow Y_0+Y_1$ as in Theorem \ref{gagliardo}, $\varphi\in\mathcal P$ with $\lim_{t\rightarrow0^+}\varphi_1(t)=0=\lim_{t\rightarrow+\infty}\frac{\varphi_1(t)}{t}$, and $(x_i)_{i=1}^n\subset X_0+X_1$ such that $\sum_{i=1}^n|x_i|\leq\varphi(u_0,u_1)$, where $u_i\in X_i$ with $\|u_i\|_{X_i}\leq1$ for $i=0,1$, then there exist sequences $(x_i^{m})_{m\in\mathbb N}$ for $1\leq i\leq n$ satisfying:
\begin{enumerate}
\item[(i)] $|x_i^m|\leq|x_i|$ for every $m\in\mathbb N$, $1\leq i\leq n$,
\item[(ii)] $\bigvee_{i=1}^n|x_i-x_i^m|\leq (u_0\vee u_1) a_m$ for certain $a_m\in\mathbb R_+$ with $a_m\underset{m\rightarrow\infty}\longrightarrow0$,
\item[(iii)] $\sup_m\Big\|\bigvee_{i=1}^n|Tx_i^m|\Big\|_{\varphi(Y_0,Y_1)}\leq\gamma \max\{\rho_{\infty,1}(T|_{X_0}),\rho_{\infty,1}(T|_{X_1})\}$.
\end{enumerate}
\end{lemma}

\begin{proof}
By Lemma \ref{l:quasi-concave}, for any $q>1$ there exist $M,N\in\mathbb N\cup\{\infty\}$, an increasing sequence $(t_k)_{k=-2M}^{2N}\subset[0,+\infty]$, and $(\varepsilon_k)_{k=-M}^N$ such that, for every $s,t\in(0,+\infty)$ we have
\begin{equation}\label{sum phi min}
\sum_{k=-M}^N\varphi_1(t_{2k+1})\min\Big(s,\frac{t}{t_{2k+1}}\Big)\leq\frac{q+1}{q-1}\varphi(s,t),
\end{equation}
and for $t\in[t_{2k}-\varepsilon_k,t_{2k+2}+\varepsilon_k]$
\begin{equation}\label{star}
\varphi_1(t)\leq q\varphi_1(t_{2k+1})\min\Big(1,\frac{t}{t_{2k+1}}\Big).
\end{equation}

Let us consider the ideal generated by $u_0\vee u_1$ in $X_0+X_1$. As usual we can consider a compact Hausdorff space $\Omega$ and a lattice homomorphism $J:C(\Omega)\rightarrow X_0+X_1$ such that $J(B_{C(\Omega)})=[-u_0\vee u_1,u_0\vee u_1]$. Since
$$
|x_i|\leq\sum_{i=1}^n|x_i|\leq\varphi(u_0,u_1)\leq u_0\vee u_1,
$$
there exist $(f_i)_{i=1}^n,h_0,h_1\in B_{C(\Omega)}$ such that $J(f_i)=x_i,J(h_0)=u_0,$ and $J(h_1)=u_1$.

Let $m\in\mathbb{N}$, and for $|k|\leq m$ let us consider the sets
$$
U_k=\{\omega\in \Omega:(t_{2k}-\varepsilon_k)h_0(\omega)<h_1(\omega)<(t_{2k+2}+\varepsilon_k)h_0(\omega)\},
$$
and
$$
V_m=\Omega\backslash\{\omega\in \Omega:t_{-2m}h_0(\omega)\leq h_1(\omega)\leq t_{2m+2}h_0(\omega)\}.
$$
Clearly, these are open subsets of $\Omega$ satisfying
$$
\Omega=V_m\cup\bigcup_{|k|\leq m} U_k.
$$
Therefore, we can consider a continuous partition of unity associated to this open covering, that is, $(\psi_k)_{|k|\leq m}$ and $\xi_m$  positive elements in $C(\Omega)$ such that for each $|k|\leq m$, $\psi_k$ is supported within $U_k$, $\xi_m$ is supported in $V_m$, and for every $\omega\in \Omega$ we have
$$
\sum_{|k|\leq m}\psi_k(\omega)+\xi_m(\omega)=1.
$$

Let us consider
$$
f_i^m=\sum_{|k|\leq m}f_i\psi_k\in C(\Omega).
$$
And denote $x_i^m=J(f_i^m)$, $y_i^k=J(f_i\psi_k)$ for $|k|\leq m$. These obviously satisfy $|y_i^k|,|x_i^m|\leq |x_i|$, for every $1\leq i\leq n$, $m\in \mathbb{N}$ and $|k|\leq m$, and
$$
x_i^m=\sum_{|k|\leq m}y_i^k.
$$

We claim that $(x_i^m)$ satisfy properties (ii) and (iii).
\medskip

In order to prove (ii), given $m\in\mathbb N$, let us consider the sets
\begin{eqnarray*}
W_1^m&=&\{\omega\in \Omega:\,h_1(\omega)<(t_{-2m}+\frac{\varepsilon_m}{2})h_0(\omega)\},\\
W_2^m&=&\{\omega\in \Omega:\,(t_{2m+2}-\frac{\varepsilon_{m+1}}{2})h_0(\omega)<h_1(\omega)\},\\
W_3^m&=&\big\{\omega\in \Omega:\,t_{-2m}h_0(\omega)<h_1(\omega)<t_{2m+2}h_0(\omega)\big\}.
\end{eqnarray*}
Since $h_0$ and $h_1$ cannot vanish simultaneously (because $h_0\vee h_1=1$), for every $m\in\mathbb N$, these open sets $W_i^m$ are such that $\bigcup_{l=1}^3 W_l^m=\Omega$. Let $(\vartheta^m_l)_{l=1,2,3}$ denote a continuous partition of unity associated to these sets, that is $\vartheta^m_l\in C(\Omega)$ with each $\vartheta^m_l$ being positive and supported in $W_l^m$, and for every $\omega\in \Omega$, and every $m\in\mathbb N$,
$$
\sum_{l=1}^3 \vartheta^m_l(\omega)=1.
$$

Note that for $1\leq i\leq n$,
$$
|(f_i-f_i^m)(\omega)|=|f_i\xi_m(\omega)|=|f_i\xi_m\Big(\sum_{l=1}^3 \vartheta^m_l\Big)(\omega)|,
$$
and since $\xi_m$ is supported in $V_m\subset \Omega\backslash W_3^m$, we have
$$
|f_i-f_i^m|\leq|f_i\xi_m\vartheta^m_1|+|f_i\xi_m\vartheta^m_2|.
$$
For $\omega\in \Omega$, we have
\begin{eqnarray}\label{eq:1}
|f_i\xi_m\vartheta^m_1(\omega)|&\leq&\varphi(h_0,h_1)\xi_m\vartheta^m_1(\omega)\\
\nonumber &\leq&\varphi(h_0(\omega),(t_{-2m}+\frac{\varepsilon_m}{2})h_0(\omega))\\
\nonumber &=& h_0(\omega)\varphi_1(t_{-2m}+\frac{\varepsilon_m}{2}).
\end{eqnarray}
Similarly, we have
\begin{eqnarray}\label{eq:2}
|f_i\xi_m\vartheta^m_2(\omega)|&\leq&\varphi(h_0,h_1)\xi_m\vartheta^m_2(\omega)\\
\nonumber &\leq&\varphi\Big(\frac{h_1(\omega)}{t_{2m+2}-\frac{\varepsilon_{m+1}}{2}},h_1(\omega)\Big)\\
\nonumber &= &h_1(\omega)\frac{\varphi_1(t_{2m+2}-\frac{\varepsilon_{m+1}}{2})}{t_{2m+2}-\frac{\varepsilon_{m+1}}{2}}.
\end{eqnarray}
Therefore, setting
$$
a_m=\varphi_1(t_{-2m}+\frac{\varepsilon_m}{2})+ \frac{\varphi_1(t_{2m+2}-\frac{\varepsilon_{m+1}}{2})}{t_{2m+2}-\frac{\varepsilon_{m+1}}{2}},
$$
and putting together the estimates \eqref{eq:1} and \eqref{eq:2} we get
$$
|x_i-x_i^m|\leq (u_0\vee u_1) a_m.
$$
The hypotheses on $\varphi_1$ clearly yield that $a_m\rightarrow0$ as $m\rightarrow\infty$, so this proves (ii).
\medskip

Finally, to prove (iii), note that by inequality (\ref{star}), for every $|k|\leq m$ and $\omega\in \Omega$ we have
\begin{eqnarray*}
\sum_{i=1}^n|f_i\psi_k(\omega)|&\leq&|\varphi(h_0(\omega),h_1(\omega))\psi_k(\omega)|\\
&\leq &h_0(\omega)\varphi_1(t_{2k+2}+\varepsilon_k)\psi_k(\omega)\\
&\leq &q\varphi_1(t_{2k+1})h_0(\omega)\psi_k(\omega),
\end{eqnarray*}
and similarly
$$
\sum_{i=1}^n|f_i\psi_k(\omega)|\leq q\frac{\varphi_1(t_{2k+1})}{t_{2k+1}}h_1(\omega)\psi_k(\omega).
$$
Therefore, the functions
$$
F^m_0=\sum_{|k|\leq m}\frac{1}{\varphi_1(t_{2k+1})}\sum_{i=1}^n|f_i\psi_k|,\hspace{1cm} F^m_1=\sum_{|k|\leq m}\frac{t_{2k+1}}{\varphi_1(t_{2k+1})}\sum_{i=1}^n|f_i\psi_k|
$$
satisfy $F^m_j\leq q h_j$ for $j=0,1.$

Now, let us consider
$$
G^m_0=\max_{|k|\leq m,1\leq i\leq n}\Big\{\frac{1}{\varphi_1(t_{2k+1})}|Ty_i^k|\Big\}
$$
in $Y_0+Y_1$. Since $T|_{X_0}:X_0\rightarrow Y_0$ is $(\infty,1)$-regular, we have

\begin{eqnarray*}
\|G^m_0\|_{Y_0}&=&\big\|\max_{|k|\leq m,1\leq i\leq n}\Big\{\frac{1}{\varphi_1(t_{2k+1})}|Ty_i^k|\Big\}\big\|_{Y_0}\\
&\leq&\rho_{\infty,1}(T|_{X_0})\Big\|\sum_{|k|\leq m}\frac{1}{\varphi_1(t_{2k+1})}\sum_{i=1}^n|y_i^k|\Big\|_{X_0}\\
&\leq&\rho_{\infty,1}(T|_{X_0})\|q u_0\|_{X_0}\\
&\leq& q \rho_{\infty,1}(T|_{X_0})
\end{eqnarray*}

While for
$$
G^m_1=\max_{|k|\leq N,1\leq i\leq n}\Big\{\frac{t_{2k+1}}{\varphi_1(t_{2k+1})}|Ty_i^k|\Big\},
$$
a similar argument yields
\begin{equation*}
\|G^m_1\|_{Y_1}\leq q \rho_{\infty,1}(T|_{X_1}).
\end{equation*}

Now, by equation (\ref{sum phi min}), we have
\begin{eqnarray*}
\max_{1\leq i\leq n}|Tx_i^m|&\leq&\max_{1\leq i\leq n}\sum_{|k|\leq m} |Ty_i^k|\\
		&\leq&\sum_{|k|\leq m} \varphi_1(t_{2k+1})\min(G^m_0,\frac{1}{t_{2k+1}}G^m_1)\\
           	&\leq&\frac{q+1}{q-1}\varphi(G^m_0,G^m_1).
\end{eqnarray*}
From this inequality, and the fact that $\|G^m_j\|_{Y_j}\leq q \rho_{\infty,1}(T|_{X_j})$ for $j=0,1$, it follows that
$$
\|\max_{1\leq i\leq n}|Tx_i^m|\|_{\varphi(Y_0,Y_1)}\leq \frac{q(q+1)}{q-1}\max\{\rho_{\infty,1}(T|_{X_0}),\rho_{\infty,1}(T|_{X_1})\}.
$$
This finishes the proof of (iii).
\end{proof}

\begin{remark}
Optimizing the estimate obtained in the previous proof for $q>1$ we could take $\gamma=3+2\sqrt{2}$.
\end{remark}

\begin{proof}[Proof of Theorem \ref{gagliardo}]

Let $R=\max\{\rho_{\infty,1}(T|_{X_0}),\rho_{\infty,1}(T|_{X_1})\}$. First, we claim that there is $K>0$ such that given $(x_i)_{i=1}^n\subset X_0+X_1$,
\begin{equation}\label{claimfi}
\text{if }\Big\|\sum_{i=1}^n |x_i|\Big\|_{\varphi(X_0,X_1)}\leq1,\text{ then } \Big\|\bigvee_{i=1}^n |Tx_i|\Big\|_{\varphi^c(Y_0,Y_1)}\leq K R.
\end{equation}

Indeed, as before let $\varphi_1(t)=\varphi(1,t)$. Without loss of generality we can assume that $\varphi_1$ is a concave function (cf. \cite[Corollary 3.1.4]{BK}). Notice that if $\lim_{t\rightarrow0^+}\varphi_1(t)=0=\lim_{t\rightarrow\infty}\frac{\varphi_1(t)}{t}$, then the conclusion follows directly from Lemma \ref{l:approximation}. Otherwise, let us consider
\begin{equation}\label{eq:decomp}
\phi_1(s)=\lim_{t\rightarrow0^+}\varphi_1(t)\vee s\lim_{t\rightarrow\infty}\frac{\varphi_1(t)}{t},\hspace{1cm}\textrm{ and }\hspace{1cm} \eta_1=\varphi_1-\phi_1.
\end{equation}
Note that, as $\phi_1$ is clearly convex, it follows that $\eta_1$ is a concave function which moreover satisfies $\lim_{t\rightarrow0^+}\eta_1(t)=0=\lim_{t\rightarrow\infty}\frac{\eta_1(t)}{t}$.

Now, if we consider $\phi(s,t)=s\phi_1\big(\frac{t}{s}\big)$ and $\eta(s,t)=s\eta_1\big(\frac{t}{s}\big)$, it follows that
\begin{equation}\label{eq:fisum}
\phi(X_0,X_1)+\eta(X_0,X_1)=\varphi(X_0,X_1)
\end{equation}
with equivalent norms (with a constant not greater than 2).

Take $(x_i)_{i=1}^n\in \varphi(X_0,X_1)$ such that $\|\sum_{i=1}^n |x_i|\|_{\varphi(X_0,X_1)}<1$, hence $\sum_{i=1}^n|x_i|\leq\varphi(u_0,u_1)$ for some $u_i\in X_i$ with $\|u_i\|_{X_i}\leq1$ for $i=0,1$. According to \eqref{eq:fisum} and using the Riesz decomposition property we can write $x_i=v_i+w_i$ where
$$
\sum_{i=1}^n |v_i|\leq \phi(u_0,u_1),\hspace{1cm}\text{ and }\hspace{1cm}\sum_{i=1}^n |w_i|\leq \eta(u_0,u_1).
$$
On the one hand, notice that $\phi(X_0,X_1)$ coincides, up to a $c$-equivalent norm, with $X_0$, $X_1$ or $X_0+X_1$ for some $c>0$.  Hence, by Lemma \ref{l:sumregular} we have that
\begin{equation}\label{eq:phiok}
\Big\|\bigvee_{i=1}^n|Tv_i|\Big\|_{\phi(Y_0,Y_1)}\leq 2 R c.
\end{equation}
On the other hand, by Lemma \ref{l:approximation} there exist a constant $\gamma$, and sequences $(w_i^{m})_{m\in\mathbb N}$ for $1\leq i\leq n$, such that
\begin{equation}\label{eq:bound}
\sup_m\Big\|\bigvee_{i=1}^n|Tw_i^m|\Big\|_{\eta(Y_0,Y_1)}\leq\gamma R
\end{equation}
and for every $i=1,\ldots,n$ and some $(a_m)_{m\in\mathbb N}$ with $a_m\underset{m\rightarrow\infty}\longrightarrow0$,
\begin{equation}\label{eq:approx}
|w_i^m-w_i|\leq (u_0\vee u_1) a_m.
\end{equation}
Note, in particular, \eqref{eq:approx} implies that
$$
\max_{1\leq i\leq n}\|v_i+w_i^m-x_i\|_{X_0+X_1}\underset{m\rightarrow\infty}\longrightarrow0,
$$
and also that
$$
\Big\|\bigvee_{i=1}^n|Tv_i+Tw_i^m|-\bigvee_{i=1}^n|Tx_i|\Big\|_{Y_0+Y_1}\underset{m\rightarrow\infty}\longrightarrow 0.
$$
While, putting together \eqref{eq:phiok} and \eqref{eq:bound}, we have
\begin{equation}
\Big\|\bigvee_{i=1}^n|Tv_i+Tw_i^m|\Big\|_{\varphi(Y_0,Y_1)}\leq\Big\|\bigvee_{i=1}^n|Tv_i|\Big\|_{\phi(Y_0,Y_1)}+\Big\|\bigvee_{i=1}^n|Tw_i^m|\Big\|_{\eta(Y_0,Y_1)}\leq (2+\gamma) R
\end{equation}
This proves claim \eqref{claimfi}.

Using the fact that $T:X_0+X_1\rightarrow Y_0+Y_1$ is bounded, the following density argument will finish the proof. Given $(x_i)_{i=1}^n\subset X_0+X_1$ with $\|\sum_{i=1}^n|x_i|\|_{\varphi^c(X_0,X_1)}<1$, we can find $(x^m)_{m\in\mathbb N}\subset X_0+X_1$ such that
$$
\sup_m\|x^m\|_{\varphi(X_0,X_1)}<1,\hspace{5mm} \text{and}\hspace{5mm}\|x^m-\sum_{i=1}^n|x_i|\|_{X_0+X_1}\rightarrow 0.
$$
Without loss of generality, we can write $x^m=\sum_{i=1}^n |x_i^m|$ for some $(x_i^m)_{m\in \mathbb N}$ such that $\sum_{i=1}^n |x_i^m|\leq\sum_{i=1}^n|x_i|$ and $\|x_i^m-x_i\|_{X_0+X_1}\rightarrow0 $ for every $i=1,\ldots,n$. By claim \eqref{claimfi}, it follows that for every $m\in \mathbb N$, $(T x^m_i)_{i=1}^n\subset \varphi^c(Y_0,Y_1)$ with
\begin{equation}\label{eq:boundap}
\Big\|\bigvee_{i=1}^n|Tx_i^m|\Big\|_{\varphi(Y_0,Y_1)}\leq\gamma R.
\end{equation}
Now, since $T:X_0+X_1\rightarrow Y_0+Y_1$ is bounded, we have that for every $i=1,\ldots,n$, $\|Tx_i^m-Tx_i\|_{Y_0+Y_1}\rightarrow 0$, and in particular we have that
\begin{equation}
\Big\|\bigvee_{i=1}^n|Tx_i^m|-\bigvee_{i=1}^n|Tx_i|\Big\|_{Y_0+Y_1}\rightarrow0.
\end{equation}
This shows that
$$
\Big\|\bigvee_{i=1}^n|Tx_i|\Big\|_{\varphi^c(Y_0,Y_1)}\leq \gamma R
$$
and finishes the proof.
\end{proof}

\begin{remark}
The proof given here is heavily motivated by the one in \cite[Theorem 4.3.11]{BK} and follows a similar approach. Actually, under the assumptions of Theorem \ref{gagliardo}, the proof of \cite[Theorem 4.3.11]{BK} essentially shows that $T:\varphi^c(X_0,X_1)\rightarrow\varphi^c(Y_0,Y_1)$ is bounded as long as $(X_0,X_1)$ and $(Y_0,Y_1)$ are interpolation couples of Banach lattices of measurable functions on certain measure space. However, the one given here is more general since the lattices we deal with do not necessarily consist of functions over a measure space.
\end{remark}

\section{Quasi-Banach lattices with the $K_{p,q}$ property}\label{Krivine property}

An application of Grothendieck's inequality due to J. L. Krivine \cite{Krivine} (see also \cite[Theorem 1.f.14]{LT2}) yields that for any Banach lattices $E,F$, every bounded linear operator $T:E\rightarrow F$ is $(2,2)$-regular with $\rho_{2,2}(T)\leq K_G\|T\|$, where $K_G$ denotes Grothendieck's constant.

This fact was later extended by N. J. Kalton to $L$-convex quasi-Banach lattices in \cite{kalton}. Recall that a quasi-Banach lattice $E$ is $L$-convex whenever its order intervals are uniformly locally convex, that is, whenever there exists $0<\varepsilon<1$ so that if $u\in E_+$ with $\|u\|=1$ and $0\leq x_i\leq u$ (for $i=1,\ldots,n$) satisfy
$$
\frac1n(x_1+\ldots+x_n)\geq(1-\varepsilon)u,
$$
then
$$
\max_{1\leq i\leq n}\|x_i\|\geq\varepsilon.
$$

In particular, every Banach lattice is $L$-convex, and so is a quasi-Banach lattice which is for an equivalent quasi-norm the $p$-concavification of a Banach lattice. In fact every $L$-convex quasi-Banach lattice is of this kind by \cite[Theorem 2.2]{kalton}, so that $L$-convex quasi-Banach lattices are exactly Nilsson's quasi-Banach lattices {\it of type} $\mathcal C$ \cite[Definition 1.7]{Nilsson}. These include classical spaces like $L_p$, $\Lambda(W,p)$ and  $L_{p,\infty}$ for $0<p\leq\infty$. On the other hand, examples of non $L$-convex quasi-Banach lattices are the $L_p(\phi)$ spaces ($0<p<\infty$) with respect to pathological submeasures $\phi$ (see \cite{kalton,Talagrand}).

Motivated by these facts we introduce the following

\begin{definition}
A quasi-Banach lattice $F$ has the $K_{p,q}$ property with constant $C>0$, if for every quasi-Banach lattice $E$, every bounded linear operator $T:E\rightarrow F$ is $(p,q)$-regular with $\rho_{p,q}(T)\leq C\|T\|$.
\end{definition}

By \cite[Theorem 3.3]{kalton}, every $L$-convex quasi-Banach lattice has the $K_{2,2}$ property. As far as we know, it is still unknown whether the converse holds. However, $L$-convex quasi-Banach lattices constitute a large collection of spaces for which our results hold. In particular, this includes every quasi-Banach lattice $E$ such that $\ell_\infty$ is not lattice finitely representable in $E$. Also, if $F$ is an $L$-convex quasi-Banach lattice and $E$ is a quasi-Banach lattice which is linearly homeomorphic to a subspace of $F$, then $E$ is $L$-convex.

Note that if a quasi-Banach lattice has the $K_{p,q}$ property for some $p,q$, then it has the $K_{\infty,1}$ property. Let us summarize this in the following chain of implications for a quasi-Banach lattice $E$:
$$
\textrm{locally convex} \Rightarrow \, L-\textrm{convex}\,  \Rightarrow  \,  K_{2,2}\,\textrm{property} \, \Rightarrow \, \, K_{\infty,1}\,\textrm{property.}
$$

We will focus now on the $K_{\infty,1}$ property for a quasi-Banach lattice, which is the weakest among the above properties.

\begin{proposition}\label{prop Kinfty1}
For a quasi-Banach lattice $E$, the following are equivalent:
\begin{enumerate}
\item[(1)] $E$ has the $K_{\infty,1}$ property with constant $C$.
\item[(2)] Every operator $T:\ell_\infty\rightarrow E$ is $(\infty,1)$-regular with  $\rho_{\infty,1}(T)\leq C\|T\|$.
\item[(3)] For every $(x_i)_{i=1}^n\subset E$ we have
$$
\big\|\max_{1\leq i\leq n}|x_i|\big\|\leq C\max_{|a_i|\leq1}\Big\|\sum_{i=1}^n a_ix_i\Big\|.
$$
\end{enumerate}
\end{proposition}

\begin{proof}
$(1)\Rightarrow(2)$ is trivial. Suppose $(2)$ holds, then given $(x_i)_{i=1}^n\subset E$, let $T:\ell_\infty\rightarrow E$ be the operator defined by
$$
T(a_i)=\sum_{i=1}^n a_ix_i,
$$
for $(a_i)_{i=1}^\infty\in\ell_\infty$. Let $e_i\in\ell_\infty$ denote the sequence having 1 in the i-th position and 0 elsewhere. By hypothesis, the operator $T$ is $(\infty,1)$-regular with $\rho_{\infty,1}(T)\leq C\|T\|$, which in particular yields
$$
\big\|\max_{1\leq i\leq n}|x_i|\big\|=\big\|\max_{1\leq i\leq n}|Te_i|\big\|\leq C\|T\|\Big\|\sum_{i=1}^n|e_i|\Big\|=C\max_{|a_i|\leq1}\Big\|\sum_{i=1}^n a_ix_i\Big\|.
$$
Therefore, $(3)$ holds.

For the implication $(3)\Rightarrow(1)$, if $F$ is a quasi-Banach lattice and $T:F\rightarrow E$ is bounded, then
$$
\big\|\max_{1\leq i\leq n}|Tx_i|\big\|\leq C\max_{|a_i|\leq1}\Big\|\sum_{i=1}^n a_iTx_i\Big\|\leq C\|T\|\Big\|\sum_{i=1}^n |x_i|\Big\|.
$$
Hence, $\rho_{\infty,1}(T)\leq C\|T\|.$
\end{proof}

A modification of \cite[Example 3.5]{kalton} provides an example of a quasi-Banach lattice without the $K_{\infty,1}$ property:

\begin{example}\label{no Kinfty1}
For each $n\in \mathbb{N}$, let $\Omega_n$ be the unit sphere in $\ell_\infty^n$, that is $\Omega_n=\{v\in\mathbb R^n:\max_{1\leq i\leq n}|v_i|=1\}$. Let $\mathcal{A}_n$ denote the algebra of all subsets of $\Omega_n$. For $u\in\mathbb{R}^n\backslash\{0\}$, let
$$
B_u=\{v\in \Omega_n:\sum_{i=1}^n u_iv_i\neq0\}.
$$
Let us consider the normalized submeasure defined, for $A\in\mathcal A_n$, by
$$
\phi_n(A)=\frac1n\inf\big\{\#S:A\subset\bigcup_{u\in S} B_u\big\}.
$$
Given $0<p<1$, consider the quasi-Banach lattice $L_p(\Omega_n,\mathcal A_n,\phi_n)$ which is the completion of the simple $\mathcal A_n$-measurable functions $f:\Omega_n\rightarrow \mathbb R$, with respect to the quasi-norm
$$
\|f\|_p=\Big(\int_0^\infty\phi_n(|f|\geq t^\frac1p)dt\Big)^\frac1p.
$$

Now, for $1\leq i\leq n$, let $f_i:\Omega_n\rightarrow \mathbb R$ be given by $f_i(v)=v_i$. It is clear that $\max_{1\leq i\leq n}|f_i(v)|=1$ for every $v\in\Omega_n$, thus
$$
\|\max_{1\leq i\leq n}|f_i|\|_p=1.
$$
On the other hand, for $a\in\mathbb R^n$ with $|a_i|\leq 1$ we have
$$
\Big|\sum_{i=1}^n a_if_i\Big|\leq n \chi_{B_a}.
$$
Therefore, we have
$$
\Big\|\sum_{i=1}^n a_if_i\Big\|_{p}\leq n^{1-\frac1p}.
$$
Taking $E$ to be the $\ell_\infty$-product of the spaces $L_p(\Omega_n,\mathcal A_n,\phi_n)$ for $n\in\mathbb N$, by Proposition \ref{prop Kinfty1}, we see that $E$ cannot have the $K_{\infty,1}$ property.
\end{example}

\section{Interpolation functors}

A direct consequence of Theorem \ref{gagliardo} yields that the functor $\varphi^c$ is an interpolation functor in the category of quasi-Banach lattices with the $K_{\infty,1}$ property:

\begin{corollary}\label{c:varphi^c}
If $(X_0,X_1)$ and $(Y_0,Y_1)$ are compatible pairs of quasi-Banach lattices such that $Y_0$ and $Y_1$ have the $K_{\infty,1}$ property, then for every $T:(X_0,X_1)\rightarrow (Y_0,Y_1)$ and every function $\varphi\in\mathcal P$, we have that $T:\varphi^c(X_0,X_1)\rightarrow \varphi^c(Y_0,Y_1)$.
\end{corollary}

\begin{proof}
Let $(X_0,X_1)$, $(Y_0,Y_1)$ be compatible couples of quasi-Banach lattices such that $Y_0$ and $Y_1$ have the $K_{\infty,1}$ property. Let $T:X_0+X_1\rightarrow Y_0+Y_1$ be an operator which is bounded as an operator $T|_{X_0}:X_0\rightarrow Y_0$ and $T|_{X_1}:X_1\rightarrow Y_1$. It follows that $T|_{X_i}$ are $(\infty,1)$-regular for $i=0,1$ so Theorem \ref{gagliardo} yields that $T:\varphi^c(X_0,X_1)\rightarrow \varphi^c(Y_0,Y_1)$ is $(\infty,1)$-regular, so in particular it is bounded and moreover
\begin{eqnarray*}
\|T|_{\varphi^c(X_0,X_1)}\|&\leq&\rho_{\infty,1}(T|_{\varphi^c(X_0,X_1)})\leq\max\{\rho_{\infty,1}(T|_{X_0}),\rho_{\infty,1}(T|_{X_1})\}\\
&\leq& C\max\{\|T|_{X_0}\|,\|T|_{X_1}\|\},
\end{eqnarray*}
where $C>0$ only depends on the $K_{\infty,1}$ constants of $Y_0$ and $Y_1$.
\end{proof}

Recall that given $(X_0,X_1)$ we can also consider $\varphi^0(X_0,X_1)$ the closure of the intersection $X_0\cap X_1$ in $\varphi(X_0,X_1)$. Our aim is to show that this is also an interpolation functor. We will need some technicalities first:

\begin{definition}
A function $\varphi\in\mathcal P$ is called doubly bounded provided there exists $C>0$ such that $\varphi_i(t)\leq C$ for $i=0,1$.
\end{definition}

\begin{lemma}\label{l:doublybounded}
A function $\varphi\in\mathcal P$ is doubly bounded if and only if $\varphi(s,t)\approx\min(s,t)$.
\end{lemma}

\begin{proof}
Suppose that there is $C>0$ such that for every $t\in\mathbb R_+$, we have $\varphi_0(t),\varphi_1(t)\leq C$. In this case, we get that
\begin{eqnarray*}
\varphi(s,t)=s\varphi_1(t/s)\leq Cs,\\
\varphi(s,t)=t\varphi_0(s/t)\leq Ct.
\end{eqnarray*}
Hence, it follows that $\varphi(s,t)\leq C\min(s,t)$. Since, for $\varphi\in\mathcal P$ we have the trivial estimate $\varphi(s,t)\geq\varphi(1,1)\min(s,t)$, the conclusion follows. The converse implication is clear.
\end{proof}

\begin{lemma}\label{varphi intersection}
Let $(X_0,X_1)$ be an interpolation couple of quasi-Banach lattices, and let $\varphi\in\mathcal P$. If $\varphi$ is not doubly bounded, and $\varphi_1(t)\rightarrow 0$ as $t\rightarrow 0$, then there is $C_{\varphi,\overline X}>0$, depending only of $\varphi$ and the quasi-norm constants of $X_0,X_1$, such that for every positive $x\in X_0\cap X_1$ with $\|x\|_{\varphi(X_0,X_1)}<1$, there exist positive $f,g\in X_0\cap X_1$ with $\|f\|_{X_0},\|g\|_{X_1}\leq C_\varphi$ and $x=\varphi(f,g)$.
\end{lemma}

\begin{proof}
By symmetry of the argument, we can suppose without loss of generality that $\lim_{t\rightarrow\infty}\varphi_0(t)=\infty$. 

Hence, for every $\delta>0$, there is $N>0$ such that $\varphi_0(\frac{N}{\delta})\geq\frac{1}{\delta}$, or in other words, $\varphi(N,\delta)\geq 1$.

Assume that $x\in (X_0\cap X_1)^+$ with $\|x\|_{\varphi(X_0,X_1)}<1$, and let $u\in X_0^+$, $v\in X_1^+$ with $\|u\|_{X_0}<1$, $\|v\|_{X_1}<1$ and
$$
x\leq\varphi(u,v).
$$

Let $C_{X_1}$ be the quasi-norm constant of $X_1$, and $\delta>0$ be small enough so that $\|v\vee \delta x\|_{X_1}<C_{X_1}$, and let $N>0$ such that $\varphi(N,\delta)\geq1$. Let $u'=u\wedge Nx$ and $v'=v\vee \delta x$. Note that $u'\in X_0\cap X_1$, $\|u'\|_{X_0}<1$, and $v'\in X_1$, $\|v'\|_{X_1}<C_{X_1}$. Moreover,
$$
\varphi(u',v')=\varphi(u,v')\wedge \varphi(Nx, v')\geq\varphi(u,v)\wedge\varphi(Nx,\delta x)= x\wedge \varphi(N,\delta) x\geq x.
$$

We distinguish two cases:\\
(a) If now we also have that $\lim_{t\rightarrow\infty} \varphi_1(t)=\infty$, then we can proceed in a similar way as before exchanging the roles of the variables in $\varphi$: let $0<\varepsilon<N$ be small enought so that $\|u'\vee \varepsilon x\|_{X_0}<1$, and let $M>0$ such that $\varphi(\varepsilon, M)\geq1$. Then, take $u^{\prime\prime}=u'\vee \varepsilon x$ and $v^{\prime\prime}=v'\wedge Mx$ which also satisfy $u^{\prime\prime},v^{\prime\prime}\in X_0\cap X_1$ with $\|u^{\prime\prime}\|_{X_0}<C_{X_0}$, $\|v^{\prime\prime}\|_{X_1}<C_{X_1}$ and $x\leq \varphi(u^{\prime\prime},v^{\prime\prime})$. 

Moreover,
$$
\varphi(u^{\prime\prime},v^{\prime\prime})\leq \varphi(Nx, Mx,)\leq \varphi(N,M) x.
$$
Consequently, we can consider $J_0(x)$ the (non-closed) ideal generated by $x$, which can be considered as a $C(\Omega)$ space for some compact Hausdorff space $\Omega$. Thus, we can consider the functions $\hat{u^{\prime\prime}}, \hat{v^{\prime\prime}}, \hat{y}\in C(\Omega)$ corresponding respectively to $u^{\prime\prime}, v^{\prime\prime}$ and $y=\varphi(u^{\prime\prime}, v^{\prime\prime})$. Recall that in this correspondence $x$ ir represented by $\hat x=1\!\!1_\Omega$, so 
$$
\hat{y}\geq \hat{x}=1\!\!1_\Omega.
$$
Thus, $\frac{1}{\hat{y}}\in C(\Omega)$ with $\|\frac{1}{\hat{y}}\|\leq1$. Set $\hat{f}=\frac{u^{\prime\prime}}{\hat{y}}$, and $\hat{g}=\frac{v^{\prime\prime}}{\hat{y}}$, which clearly correspond to elements $f,g\in J_0(x)$ such that 
$$
\varphi(f,g)=x.
$$
This identity follows from the fact that
$$
\varphi(\hat{f},\hat{g})=\varphi(\frac{u^{\prime\prime}}{\hat{y}},\frac{v^{\prime\prime}}{\hat{y}}=\frac{\varphi(u^{\prime\prime},v^{\prime\prime})}{\hat{y}}=1\!\!1_\Omega=\hat{x}.
$$
Moreover, we have
$$
f\leq u^{\prime\prime}\leq Nx,\quad g\leq v^{\prime\prime}\leq Mx.
$$
Hence, $f.g\in X_0\cap X_1$, with $\|f\|_{X_0}\leq \|u^{\prime\prime}\|_{X_0}<C_{X_0}$ and $\|g\|_{X_1}\leq \|v^{\prime\prime}\|_{X_1}<C_{X_1}$.
\medskip

\noindent (b) If on the contrary, $\varphi_1$ is bounded, then set $C_\varphi=\sup_{s>0}\varphi_1(s)<\infty$, so that 
$$
\varphi(s,t)=s\varphi_1(\frac{t}{s})\leq C_\varphi s.
$$
Since $x=\varphi(u',v')$, we have $x\leq C_\varphi u'$ and 
$$
x=\varphi(x,x)\leq \varphi(C_\varphi u',x).
$$
On the other hand, $x=\varphi(u',v')\leq \varphi(C_\varphi u',v')$ (assuming without loss of generality that $C_\varphi \geq1$). Thus,
$$
x\leq \varphi(C_\varphi u',x\wedge v').
$$
Then, we can take $u^{\prime\prime}=Cu'$ and $v^{\prime\prime}=x\wedge v'$. Then $u^{\prime\prime},v^{\prime\prime}$ belong to $J_0(x)$, the (non-closed) ideal generated by $x$, which correspond to the space $C(\Omega)$, and satisfy
$$
\|u^{\prime\prime}\|_{X_0}\leq C_\varphi ,\quad \|v^{\prime\prime}\|_{X_1}<C_{X_1}.
$$
Hence, as before we may find $f\leq u^{\prime\prime}$ and $g\leq v^{\prime\prime}$ with $x=\varphi(f,g)$.
\end{proof}

This fact will allow us to show that $\varphi^0$ is an interpolation functor in the category of quasi-Banach lattices with the $K_{\infty,1}$ property. More precisely:

\begin{theorem}
Let $(X_0,X_1)$ and $(Y_0,Y_1)$ be compatible pairs of quasi-Banach lattices and $T:X_0+X_1\rightarrow Y_0+Y_1$ such that $T|_{X_j}:X_j\rightarrow Y_j$ is $(\infty,1)$-regular for $j=0,1$. Then for  every function $\varphi\in\mathcal P$, we have that $T:\varphi^0(X_0,X_1)\rightarrow \varphi^0(Y_0,Y_1)$ is $(\infty,1)$-regular with
$$
\rho_{\infty,1}(T|_{\varphi^0(X_0,X_1)})\leq C\max\{\rho_{\infty,1}(T|_{X_0}),\rho_{\infty,1}(T|_{X_1})\},
$$
for some $C>0$ depending only on $X_0$, $X_1$, $Y_0$, $Y_1$ and $\varphi$.
\end{theorem}

\begin{proof}
If $\varphi$ is doubly bounded, by Lemma \ref{l:doublybounded}, it follows that $\varphi^0(X_0,X_1)=X_0\cap X_1$ (with an equivalent norm). Therefore, in this case the conclusion follows. 

Note that we can consider a decomposition as the one given in \eqref{eq:decomp}:
\begin{equation}
\phi_1(s)=\lim_{t\rightarrow0^+}\varphi_1(t)\vee s\lim_{t\rightarrow\infty}\frac{\varphi_1(t)}{t},\hspace{1cm}\textrm{ and }\hspace{1cm} \eta_1=\varphi_1-\phi_1.
\end{equation}
As before, note that $\phi_1$ is convex, so $\eta_1$ is concave. Thus, taking $\phi(s,t)=s\phi_1\big(\frac{t}{s}\big)$ and $\eta(s,t)=s\eta_1\big(\frac{t}{s}\big)$, it holds that
\begin{equation}\label{eq:decompfi0}
\varphi=\phi+\eta
\end{equation}
where $\phi(s,t)\approx\max(s,t)$ and $\lim_{t\rightarrow0}\eta_1(t)=0=\lim_{t\rightarrow\infty}\frac{\eta_1(t)}{t}.$

Let $(x_i)_{i=1}^n\subset X_0\cap X_1$ be positive with $\|\sum_{i=1}^n|x_i|\|_{\varphi(X_0,X_1)}<1$. Since $T x_i\in Y_0\cap Y_1$ for every $1\leq i\leq n$, it will be enough to show that
\begin{equation}\label{eq:boundfi0}
\|\max_{1\leq i\leq n}|Tx_i|\|_{\varphi(Y_0,Y_1)}\leq \gamma\max\{\rho_{\infty,1}(T|_{X_0}),\rho_{\infty,1}(T|_{X_1})\},
\end{equation}
for a certain constant $\gamma>0$ independent of $T$ and $(x_i)_{i=1}^n$.

Note that $\sum_{i=1}^n|x_i|\leq \varphi(u_0,u_1)$ with $u_j\in X_j$ and $\|u_j\|_{X_j}\leq1$. Using the Riesz decomposition property and \eqref{eq:decompfi0}, we can write $x_i=f_i+g_i$ with $0\leq f_i,g_i\leq x_i$ in $X_0\cap X_1$, such that $f_i\leq\phi(u_0,u_1)$ and $g_i\leq\eta(u_0,u_1)$.

On the one hand, since $\phi(X_0,X_1)$ coincides, up to an equivalent norm, with $X_0$, $X_1$ or $X_0+X_1$, using Lemma \ref{l:sumregular}, it follows that
\begin{equation}
\|\max_{1\leq i\leq n}|Tf_i|\|_{\phi(Y_0,Y_1)}\leq\gamma_0\max\{\rho_{\infty,1}(T|_{X_0}),\rho_{\infty,1}(T|_{X_1})\}
\end{equation}
for a certain constant $\gamma_0$. On the other hand, since we can assume that $\varphi$, and hence $\eta$, is not doubly bounded, by Lemma \ref{varphi intersection}, there exist $C_{\eta,\overline X} > 0$ and $v_0,v_1\in X_0\cap X_1$ with $\|v_j\|_{X_j}\leq C_{\eta,\overline X}$, such that 
$$
\sum_{i=1}^n |g_i|=\eta(v_0,v_1).
$$
Hence, Lemma \ref{l:approximation} applied to $(g_i)_{i=1}^n$, $v_0$ and $v_1$ provides for $1\leq i\leq n$ sequences $(g_i^m)_{m\in\mathbb N}$ in $X_0+X_1$ such that for $m\in \mathbb{N}$ we have 
$$
\max_{1\leq i\leq n}|g_i-g_i^m|\leq (v_0\vee v_1) a_m
$$ 
for certain $a_m\in\mathbb R_+$ with $a_m\underset{m\rightarrow\infty}\longrightarrow0$, and 
$$
\sup_m\Big\|\max_{1\leq i\leq n}|Tg_i^m|\Big\|_{\varphi(Y_0,Y_1)}\leq\gamma \max\{\rho_{\infty,1}(T|_{X_0}),\rho_{\infty,1}(T|_{X_1})\}.
$$

Hence, since $v_0,v_1\in X_0\cap X_1$, for every $1\leq i\leq n$, it holds that $g_i^m\rightarrow g_i$ in $X_0\cap X_1$. In particular, $Tg_i^m\rightarrow Tg_i$ also in $Y_0\cap Y_1$, which yields
\begin{equation}
\|\max_{1\leq i\leq n}|Tg_i|\|_{\eta(Y_0,Y_1)}\leq\gamma\max\{\rho_{\infty,1}(T|_{X_0}),\rho_{\infty,1}(T|_{X_1})\}.
\end{equation}
Since $Tx_i=Tf_i+Tg_i$, this finishes the proof.
\end{proof}

The above result immediately yields the following:

\begin{corollary}\label{varphi0}
If $(X_0,X_1)$ and $(Y_0,Y_1)$ are compatible pairs of quasi-Banach lattices such that $Y_0$ and $Y_1$ have the $K_{\infty,1}$ property, then for every $T:(X_0,X_1)\rightarrow (Y_0,Y_1)$ and every function $\varphi\in\mathcal P$, we have that $T:\varphi^0(X_0,X_1)\rightarrow \varphi^0(Y_0,Y_1)$.
\end{corollary}

\begin{remark}
If $X_0$ and $X_1$ are quasi-Banach lattices of measurable functions over a measure space and for some constant $M>0$ and vectors $(x_i)_{i=1}^n\subset X_j$ it holds that
\begin{equation}\label{Kinfty1rad}
\big\|\max_{1\leq i\leq n}|x_i|\big\|_{X_j}\leq M\max_{t\in[0,1]}\Big\|\sum_{i=1}^n r_i(t)x_i\Big\|_{X_j},
\end{equation}
where $r_i$ denotes the $i$-th Rademacher function, and the function $\varphi\in\mathcal P$ satisfies the condition that $\varphi(s,t)\rightarrow 0$ as $s\rightarrow 0$ or $t\rightarrow 0$, and $\varphi(s,t)\rightarrow \infty$ as $s\rightarrow \infty$ or $t\rightarrow \infty$, then \cite[Theorem 2.1]{Nilsson} asserts that $\varphi^0(X_0,X_1)$ coincides with the $\langle\cdot\rangle_\varphi$-method introduced by J. Peetre in \cite{Peetre}. Note that by Proposition \ref{prop Kinfty1}, condition \eqref{Kinfty1rad} implies the $K_{\infty,1}$ property of $X_j$. Hence, under these somehow stronger assumptions, the interpolation result of Theorem \ref{varphi0} also follows from this fact.
\end{remark}

\begin{remark}
We do not know whether the $K_{\infty,1}$ property in Corollaries \ref{c:varphi^c} and \ref{varphi0} is actually necessary.
\end{remark}

{\bf Acknowledgments.} Second author gratefully acknowledges support of Spanish MINECO through grants MTM2012-31286 and MTM2013-40985-P, as well as Grupo UCM 910346. He wishes to thank the Equipe d'Analyse Fonctionnelle of the Institut de Math\'ematiques de Jussieu for their always warm hospitality. We also thank the anonymous referees for their valuable comments.

\bibliographystyle{amsplain}

\end{document}